\documentclass[12pt]{article}
\usepackage{amsmath,amssymb,amsthm}
\usepackage{amsfonts, amsmath, amsthm, amssymb}
\usepackage{epsfig}
\usepackage{color}

\title{Signed magic rectangles with two filled cells in each column}
\author {
Abdollah Khodkar and Brandi Ellis\\
Department of Mathematics\\
University of West Georgia\\
Carrollton, GA 30118\\
{\tt akhodkar@westga.edu},
{\tt bellis5@my.westga.edu}
}

\date{}

\newtheorem{prelem}{{\bf Theorem}}

 \newtheorem{theorem}{Theorem}

\newtheorem{lemma}[theorem]{Lemma}

\newtheorem{proposition}[theorem]{Proposition}

\theoremstyle{definition}

\theoremstyle{remark}

\begin{document}

\maketitle

\begin{abstract}
\noindent A {\em signed magic rectangle} $SMR(m,n;r, s)$ is an $m \times n$ array with entries from $X$, where
$X=\{0,\pm1,\pm2,\ldots, $ $\pm (mr-1)/2\}$ if $mr$ is odd and $X = \{\pm1,\pm2,\ldots,\pm mr/2\}$ if $mr$ is even,
such that precisely $r$ cells in every row and $s$ cells in every column are filled,
every integer from set $X$ appears exactly once in the array and
the sum of each row and of each column is zero. In this paper, we prove that a signed magic rectangle
$SMR(m,n;r, 2)$ exists if and only if either $m=2$, $n\equiv 0,3 \pmod 4$ and $n=r$ or
$m,r\geq 3$ and $mr=2n$.
\end{abstract}

\section{Introduction}\label{SEC1}

A {\em magic rectangle} of order $m\times n$, $MR(m,n)$, is an arrangement of the numbers from 0 to $mn-1$ in an $m\times n$ rectangle such that each number occurs exactly once in the rectangle and the sum of the entries of each row is the same and the sum of entries of each column is also the same.
The following theorem, whose proof can be found in \cite {TH1, TH2} and \cite {sun}, settles the existence of an $MR(m,n)$.

\begin{theorem}\label{TH:sun}
An $m \times n$ magic rectangle exists if and only if $m \equiv n \pmod 2$, $m + n > 5$, and $m, n > 1$.
\end{theorem}

A {\em $k$-magic square} of order $n$ is an arrangement of the numbers from 0 to $kn-1$ in an $n\times n$ array such that each row and each column has exactly $k$ filled cells, each number occurs exactly once in the array, and the sum of the entries of any row or any column is the same. The study of magic squares with empty cells was initiated in \cite{KL1}. A $k$-magic square is called $k$-{\em diagonal} if its entries all belong to $k$ consecutive diagonals (this includes broken diagonals as well).

\begin{theorem}\label{TH:KL} \cite{KL1}
There exists a $k$-diagonal magic square of order $n$ if and only if $n=k=1$ or
$3\leq k\leq n$ and either $n$ is odd or $k$ is even.
\end{theorem}

A {\em signed magic rectangle} $SMR(m,n;r, s)$ is an $m \times n$ array with entries from $X$, where
$X=\{0,\pm1,\pm2,\ldots,\pm (mr-1)/2\}$ if $mr$ is odd and $X = \{\pm1,\pm2,\ldots,\pm mr/2\}$ if $mr$ is even,
such that precisely $r$ cells in every row and $s$ cells in every column are filled,
every integer from set $X$ appears exactly once in the array and
the sum of each row and of each column is zero.
By the definition, $mr=ns$, $r\leq n$ and $s \leq m$. If $r=n$ or $s=m$, then the rectangle has no empty cell.
In the case where $m = n$, we call the array a {\em signed magic square}.
Signed magic squares represent a type of magic square where each number from the set $X$
is used exactly once.

The following two theorems can be found in \cite{KSW}.
\begin{theorem}
	\label{TH:KSW1}
	An $SMR(m,n)$ exists precisely when $m = n = 1$, or when $m = 2$ and $n \equiv 0, 3 \pmod4$, or when $n = 2$ and $m \equiv 0, 3 \pmod4$, or when $m, n > 2$.
\end{theorem}

In \cite{KSW} the notation $SMS(n;k)$ is used for
a signed magic square with $k$ filled cells in each row and
$k$ filled cells in each column.

\begin{theorem}\label{TH:KSW2}
	There exists an $SMS(n;k)$ precisely when $n=k = 1$ or $3\leq k\leq n$.
\end{theorem}

In this paper we prove that a signed magic rectangle
$SMR(m,$ $n;r, 2)$ exists if and only if either $m=2$ and $n=r\equiv 0,3 \pmod 4$ or
$m,r\geq 3$ and $mr=2n$.

\section {Main constructions}\label{SEC2}
A rectangular array is {\em shiftable} if it contains the same number of positive entries as negative entries in every column and in every row. Figure \ref {2,4;4,2} displays a shiftable $SMR(2,4;4,2)$.
These arrays are called \textit{shiftable} because they may be shifted to use different absolute values. By increasing the absolute value of each entry by $k$, we add $k$ to each positive entry and $-k$ to each negative entry. If the number of entries in a row is $2\ell$, this means that we add $\ell k + \ell(-k) = 0$ to each row, and the same argument applies to the columns. Thus, when shifted, the array retains the same row and column sums.

\begin{figure}[ht]
$$\begin{array}{|c|c|c|c|}\hline
1&-2&-3&4\\\hline
-1&2&3&-4\\\hline
\end{array}$$
\caption{A shiftable $SMR(2,4;4,2)$}
\label{2,4;4,2}
\end{figure}

Let $A$ be an array. We write $(i,j;e)\in A$ if and only if the entry $e$ is in row $i$ and column $j$.

\begin{theorem}\label{TH:kn;kr-km,kn}
Let there exist a shiftable $SMR(m,n;r,s)$. Then for every $k\geq 1$
\begin{enumerate}
\item there exists a shiftable $SMR(m,kn;kr,s)$ and
\item there exists a shiftable $SMR(km,kn;r,s)$ .
\end{enumerate}
\end{theorem}

\begin{proof}
Let $A$ be a shiftable $SMR(m,n;r,s)$. Note that since $A$ is shiftable, it follows that $r$ and
$s$ are both even.
Partition an empty $m\times kn$ rectangle, say $B$, into
$k$ empty rectangles of size $m\times n$, say $P_{\ell}$, where $0\leq\ell\leq k-1$. For each $(i,j;e)\in A$ we fill the cell $(i,j)$ of $P_{\ell}$ with $e+\ell(mr/2)$ if $e$ is positive or with $e-\ell(mr/2)$ if $e$ is negative. The resulting rectangle is a shiftable $SMR(m,kn;kr,s)$. For an example see Figure \ref{2,12;12,2}.

We now prove that there exists a shiftable $SMR(km,kn;r,s)$ for $k\geq 1$.
Partition an empty $km\times kn$ rectangle, say $C$, into
$k^2$ empty rectangles of size $m\times n$, say $P_{a,b}$, where $0\leq a, b\leq k-1$. For each $(i,j;e)\in A$ we fill the cell $(i,j)$ of $P_{a,a}$ with $e+a(mr/2)$ if $e$ is positive or with $e-a(mr/2)$  if $e$ is negative for $0\leq a\leq k-1$. The resulting rectangle is a shiftable $SMR(km,kn;r,s)$. For an example see Figure \ref{6,12;4,2}.
\end{proof}

\begin{figure}[ht]
$$\begin{array}{|c|c|c|c||c|c|c|c||c|c|c|c|}\hline
1&-2&-3&4&5&-6&-7&8&9&-10&-11&12\\\hline
-1&2&3&-4&-5&6&7&-8&-9&10&11&-12\\\hline
\end{array}$$
\caption{A shiftable $SMR(2,12;12,2)$}
		\label{2,12;12,2}
\end{figure}

\begin{figure}[ht]
$$\begin{array}{|c|c|c|c||c|c|c|c||c|c|c|c|}\hline
1&-2&-3&4&&&&&&&&\\\hline
-1&2&3&-4&&&&&&&&\\\hline\hline
&&&&5&-6&-7&8&&&&\\\hline
&&&&-5&6&7&-8&&&&\\\hline\hline
&&&&&&&&9&-10&-11&12\\\hline
&&&&&&&&-9&10&11&-1
2\\\hline
\end{array}$$
\caption{A shiftable $SMR(6,12;4,2)$}
		\label{6,12;4,2}
\end{figure}

\begin{theorem}\label{TH:kn+n';kr+r'}
Let there exist a shiftable $SMR(m,n;r,s)$
and a (shiftable) $SMR(m,n';r',s)$ with $mr'$ even.
Then there exists a (shiftable) $SMR(m,kn+n';kr+r',s)$ for $k\geq 1$.
\end{theorem}

\begin{proof}

Apply Part 1 of Theorem \ref{TH:kn;kr-km,kn} with a shiftable $SMR(m,n;r,s)$ to obtain a shiftable $SMR(m,kn;kr,s)$, say $A$, for $k\geq 1$.
Let $B$ be a (shiftable) $SMR(m,n';r',s)$ and
let $C$ be the $m\times kn$ rectangle obtained from $A$
by adding $mr'/2$ to each positive entry of $A$ and subtracting $mr'/2$ from each negative entry of $A$.
Finally, let $D$ be the $m\times (kn+n')$ rectangle obtained from $B$ and $C$ as follows: if $(i,j;e)\in B$, then $(i,j;e)\in D$ and if $(i,j;e)\in C$, then $(i,j+n';e)\in D$.
It is easy to see that $D$ is a (shiftable) $SMR(m,kn+n';kr+r',s)$.
\end{proof}

Figure \ref{2,11;11,2} displays an $SMR(2,11;11,2)$ obtained by the construction given in the proof of Theorem \ref{TH:kn+n';kr+r'} using the shiftable $SMR(2,4;4,2)$ given in Figure \ref{2,4;4,2}, an
$SMR(2,3;3,2)$ and $k=2$.

\begin{figure}[ht]
$$\begin{array}{|c|c|c||c|c|c|c||c|c|c|c|}\hline
-1&-2&3&-4&5&6&-7&-8&9&10&-11 \\ \hline
1&2&-3&4&-5&-6&7&8&-9&-10&11 \\ \hline
\end{array}$$
\caption{An $SMR(2,11;11,2)$}
		\label{2,11;11,2}
\end{figure}

\begin{theorem}\label{TH:km+m',kn+n'}
Let there exist a shiftable $SMR(m,n;r,s)$ and a (shiftable) $SMR(m',n';r,s)$ with $m'r$ even., then there exists a (shiftable) $SMR(km+m',kn+n';r,s)$ for $k\geq 1$.
\end{theorem}

\begin{proof}
Apply Part 2 of Theorem \ref{TH:kn;kr-km,kn} with a shiftable $SMR(m,n;$ $r,s)$ to obtain a shiftable $SMR(km,kn;r,s)$, say $A$, for $k\geq 1$.
Let $B$ be a (shiftable) $SMR(m',n';r,s)$ and
let $C$ be the $m\times kn$ rectangle obtained from $A$
by adding $m'r/2$ to each positive entry of $A$ and subtracting $m'r/2$ from each negative entry of $A$.
Finally, let $D$ be the $(km+m')\times (kn+n')$ rectangle obtained from $B$ and $C$ as follows: if $(i,j;e)\in B$, then $(i,j;e)\in D$ and if $(i,j;e)\in C$, then $(i+m',j+n';e)\in D$.
It is easy to see that $D$ is a (shiftable) $SMR(km+m',kn+n';r,s)$.
\end{proof}

Figure \ref{7,14;4,2} displays a shiftable $SMR(7,14;4,2)$ obtained by the construction given in the proof of Theorem \ref{TH:km+m',kn+n'} using the shiftable $SMR(2,4;4,2)$ given in Figure \ref{2,4;4,2},
the shiftable $SMR(3,6;4,2)$ given in Figure  \ref{3,6;4,2}, and $k=2$.

\begin{figure}[ht]
$$\begin{array}{|c@{\hspace{0.3mm}}|c@{\hspace{0.3mm}}|c@{\hspace{0.3mm}}
|c@{\hspace{0.3mm}}|c@{\hspace{0.3mm}}|c@{\hspace{0.3mm}}
|c@{\hspace{0.3mm}}|c@{\hspace{0.3mm}}|c@{\hspace{0.3mm}}
|c@{\hspace{0.3mm}}|c@{\hspace{0.3mm}}|c@{\hspace{0.3mm}}
|c@{\hspace{0.3mm}}|c@{\hspace{0.3mm}}|} \hline
1&&-3&-4&&6&&&&&&&& \\ \hline
-1&2&&4&-5&&&&&&&&& \\ \hline
&-2&3&&5&-6&&&&&&&& \\ \hline\hline
&&&&&&-7&8&9&-10&&&& \\ \hline
&&&&&&7&-8&-9&10&&&& \\ \hline\hline
&&&&&&&&&&-11&12&13&-14 \\ \hline
&&&&&&&&&&11&-12&-13&14 \\ \hline

\end{array}$$
\caption{A shiftable $SMR(7,14;4,2)$}
		\label{7,14;4,2}
\end{figure}

\section {The existence of an $SMR(m,3m/2;3,$ $2)$ and an $SMR(m,5m/2;5,2)$}\label{SEC3}
In this section we present direct constructions for the existence of an $SMR(m,3m/2;3,2)$, where $m\geq 2$ and even,  and an $SMR(m,5m/2;5,2)$, where $m\geq 4$ and even. We will make use of these results in Section \ref{SEC4}.
Note that if $m$ is odd there is no $SMR(m,3m/2;3,$ $2)$ because $3m$ is odd and there is no $SMR(m,5m/2;5,2)$ because $5m$ is odd.

\begin{proposition}\label{m,3m/2;3,2}
There exists an $SMR(m,3m/2;3,2)$ for $m$ even and $m\geq 2$.
\end{proposition}

\begin{proof}
Define an $m\times 3$ rectangle $A$ as follows.

\noindent Column 1: $ \left\{\begin{array}{l}
(i,1;i)\in A \mbox { for } 1\leq i\leq m/2,\\
(i,1;(m/2)-i)\in A  \mbox { for } (m/2)+1\leq i\leq m.\\
\end{array}\right.$

\noindent Column 2: $ \left\{\begin{array}{lll}
(i,2;(3m/2)-2i+1)\in A \mbox { for } 1\leq i\leq m/2,\\
(i,2;-i)\in A  \mbox { for } (m/2)+1\leq i\leq m.\\
\end{array}\right.$

\noindent Column 3: $ \left\{\begin{array}{l}
(i,3;(-3m/2)+i-1)\in A \mbox { for } 1\leq i\leq m/2,\\
(i,3;(-m/2)+2i)\in A  \mbox { for } (m/2)+1\leq i\leq m.\\
\end{array}\right.$

\noindent By construction, it is easy to see that the entries in $A$ consist of
$\{\pm1, \pm 2, \ldots, \pm 3m/2\}$, which are the numbers in an $SMR(m,$ $3m/2;3,2)$.
Figure \ref{A.m=8.3.2} displays the rectangle $A$ when $m=8, 10$.
We now prove that the sum of each row of $A$ is zero. The row sum for row $i$ of $A$, where $1\leq i\leq m/2$, is
$$  i+ ((3m/2)-2i+1)+ ((-3m/2)+i-1)=0.$$
Similarly, the row sum for row $i$ of $A$, where $(m/2)+1\leq i\leq m$, is
$$ ((m/2)-i)+(-i)+((-m/2)+2i) =0.$$

Let $a,k$ and $-k$ be the numbers in a row of $A$. Then $a+k+(-k)=0$, which implies that $a=0$. Since zero does not appear in $A$, it follows that the numbers $k$ and $-k$ do not appear in the same row of $A$.

Now let $B$ be an empty $m\times 3m/2$ rectangle. For each $(i,j;k)\in A$ let $(i,|k|;k)\in B$.
By construction, the numbers in row $i$ of $B$ are precisely the numbers in row $i$ of $A$.
Therefore the row sum for each row of $B$ is also zero.
Since $\pm k$ are entries of $A$ for each $1\leq k\leq 3m/2$, it follows that column $k$ of $B$ contains only $k$ and $-k$. Hence, $B$ is an $SMR(m,3m/2;3,2)$ for $m$ even and $m\geq 2$.
\end{proof}

Figure \ref{8,12;3,2} displays an $SMR(8,12;3,2)$ obtained by the construction given in Proposition \ref{m,3m/2;3,2}.

\begin{figure}[ht]
$$\begin{array}{ccc}
\begin{array}{|c|c|c|}\hline
1&11&-12 \\ \hline
2&9&-11 \\ \hline
3&7&-10 \\ \hline
4&5&-9 \\ \hline
-1&-5&6 \\ \hline
-2&-6&8 \\ \hline
-3&-7&10 \\ \hline
-4&-8&12 \\ \hline
\end{array}&&
\begin{array}{|c|c|c|}\hline
1&14&-15 \\ \hline
2&12&-14 \\ \hline
3&10&-13 \\ \hline
4&8&-12 \\ \hline
5&6&-11 \\ \hline
-1&-6&7 \\ \hline
-2&-7&9 \\ \hline
-3&-8&11 \\ \hline
-4&-9&13 \\ \hline
-5&-10&15 \\ \hline
\end{array}\\
\mbox{ Array } A \mbox { when } m=8&&\mbox{ Array } A \mbox { when } m=10\\
\end{array}
$$
\caption{Array $A$ given in Proposition \ref{m,3m/2;3,2}}
		\label{A.m=8.3.2}
\end{figure}

\begin{figure}[ht]
$$\begin{array}{|c@{\hspace{1.0mm}}|c@{\hspace{1.0mm}}|c@{\hspace{1.0mm}}
|c@{\hspace{1.0mm}}||c@{\hspace{1.0mm}}|c@{\hspace{1.0mm}}
|c@{\hspace{1.0mm}}|c@{\hspace{1.0mm}}||c@{\hspace{1.0mm}}
|c@{\hspace{1.0mm}}|c@{\hspace{1.0mm}}|c@{\hspace{1.0mm}}|} \hline
1&&&&&&&&&&11&-12 \\ \hline
&2&&&&&&&9&&-11&  \\ \hline
&&3&&&&7&&&-10&&    \\ \hline
&&&4&5&&&&-9&&&    \\ \hline
-1&&&&-5&6&&&&&&  \\ \hline
&-2&&&&-6&&8&&&&  \\ \hline
&&-3&&&&-7&&&10&&  \\ \hline
&&&-4&&&&-8&&&&12  \\ \hline
\end{array}$$
\caption{An $SMR(8,12;3,2)$}
		\label{8,12;3,2}
\end{figure}

It is an easy exercise to see that there is no $SMR(2,5;5,2)$.  The following proposition
shows how to build an $SMR(m,5m/2;5,2)$ for $m$ even and $m\geq 4$.

\begin{proposition}\label{m,5m/2;5,2}
There exists an $SMR(m,5m/2;5,2)$ for $m$ even and $m\geq 4$.
\end{proposition}

\begin{proof}
Define an $m\times 5$ rectangle $C$ as follows.

\noindent Column 1: $ \left\{\begin{array}{l}
(i,1;i)\in C \mbox { for } 1\leq i\leq m/2,\\
(i,1;(m/2)-i)\in C  \mbox { for } \frac{m+2}{2}\leq i\leq m.\\
\end{array}\right.$

\noindent Column 2: $ \left\{\begin{array}{l}
(i,2;(m/2)+2i-1)\in C \mbox { for } 1\leq i\leq m/2,\\
(i,2;(-3m/2)+i-1)\in C  \mbox { for } \frac{m+2}{2}\leq i\leq m.\\
\end{array}\right.$

\noindent Column 3: $ \left\{\begin{array}{l}
(i,3;(-m)-i)\in C \mbox { for } 1\leq i\leq m/2,\\
(i,3;(5m/2)-2i+2)\in C \mbox { for } \frac{m+2}{2}\leq i\leq m.\\
\end{array}\right.$

\noindent Column 4: $ \left\{\begin{array}{l}
(i,4;(-3m/2)-i)\in C \mbox { for } 1\leq i\leq (m/2),\\
(i,4;(3m/2)+i)\in C  \mbox { for } \frac{m+2}{2}\leq i\leq m.\\
\end{array}\right.$

\noindent Column 5: $ \left\{\begin{array}{l}
(i,5;2m-i+1)\in C \mbox { for } 1\leq i\leq m/2,\\
(i,5;-3m+i-1)\in C  \mbox { for } \frac{m+2}{2}\leq i\leq m.\\
\end{array}\right.$
\vspace{3mm}

\noindent By construction, the entries in $C$ consist of
$\{\pm1, \ldots, \pm 5m/2\}$, which are the numbers in an $SMR(m,5m/2;5,2)$.
Figure \ref{C.m=8.5.2} displays the rectangle $C$ when $m=8$.
We now prove that the sum of each row of $C$ is zero. The row sum for row $i$ of $C$, where $1\leq i\leq m/2$, is
$$i+ ((m/2)+2i-1)+ (-m-i) + ((-3m/2)-i) + (2m-i+1) =0.$$
Similarly, the row sum for row $i$ of $C$, where $(m/2)+1\leq i\leq m$, is
$$\begin{array}{r}
((m/2)-i) + ((-3m/2)+i-1) + (5m/2)-2i+2)\\
+ ((3m/2)+i) + (-3m+i-1)  =0.\end{array}$$

Let $a,b,c,d,e$ be the numbers in row $i$ and columns $1,2,3,4,5$ of $C$, respectively.
It is straightforward to see that if $x,y\in \{a,b,c\}$ and $z\in \{d,e\}$, then
$x+y\neq 0$ and $x+z\neq 0$. Now let $d+e=0$. If $1\leq i\leq m/2$, then
$$d+e= ((-3m/2)-i)+ (2m-i+1)= (m/2)-2i+1=0.$$
This implies that $i=(m+2)/4$.

\noindent If $(m/2)+1\leq i\leq m$, then
$$d+e=((3m/2)+i)+(-3m+i-1)=(-3m/2)+2i-1=0.$$
This implies that $i=(3m+2)/4.$

Therefore if $m\equiv 0\pmod 4$, then the numbers $k$ and $-k$ do not appear in the same row of $C$. If
$m\equiv 2 \pmod 4$ and $i\neq (m+2)/2, (3m+2)/4$, then the numbers $k$ and $-k$ do not appear in row $i$ of $C$.

When $m\equiv 2 \pmod 4$ we construct an $m\times 5$ array $C'$ by rearranging the eight entries of $C$
which are in the intersection of columns 1 and 2 with rows $(m-2)/2, (m+2)/2, (3m-2)/4$ and $(3m+2)/4$ as
follows. Switch
$$\begin{array}{l}
((m-2)/4,1;(m-2)/4) \mbox { and } (m+2)/4,1;(m+2)/4),\\
((m-2)/4,5;(7m+6)/4) \mbox{ and } ((m+2)/4,5;(7m+2)/4),\\
((3m-2)/4,1;(-m+2)/4) \mbox{ and } (3m+2)/4,1; (-m-2)/4), \\
\mbox{and }((3m-2)/4,5;(-9m-6)/4) \mbox{ and }((3m+2)/4,5;\\
(-9m-2)/4).\\
\end{array}$$
Figure \ref{C.m=8.5.2} displays the rectangle $C'$ when $m=10$.
It is easy to see that the sum of each row of $C'$ is zero and $k$ and $-k$ do not appear in any row of $C'$.

Now let $m\equiv 0 \pmod 4$, $m\geq 4$, and let $D$ be an empty $m\times 5m/2$ rectangle. For each $(i,j;k)\in C$ let $(i,|k|;k)\in D$.
By construction, the numbers in row $i$ of $D$ are precisely the numbers in row $i$ of $C$.
Therefore the row sum for each row of $D$ is also zero.
Since $\pm k$ are entries of $C$ for each $1\leq k\leq 5m/2$, it follows that column $k$ of $D$ contains only $k$ and $-k$. Hence, $D$ is an $SMR(m,5m/2;5,2)$.

Similarly, if $m\equiv 2 \pmod 4$ and $m\geq 6$, we use the array $C'$ to build an $SMR(m,5m/2;5,2)$.
\end{proof}


\begin{figure}[ht]
$$\begin{array}{ccc}
\begin{array}{|c|c|c|c|c|}\hline
1&5&-9&-13&16 \\ \hline
2&7&-10&-14&15 \\ \hline
3&9&-11&-15&14  \\ \hline
4&11&-12&-16&13 \\ \hline
-1&-8&12&17&-20 \\ \hline
-2&-7&10&18&-19 \\ \hline
-3&-6&8&19&-18 \\ \hline
-4&-5&6&20&-17 \\ \hline
\end{array}&&
\begin{array}{|c|c|c|c|c|}\hline
1&6&-11&-16&20 \\ \hline
3&8&-12&-17&18 \\ \hline
2&10&-13&-18&19 \\ \hline
4&12&-14&-19&17 \\ \hline
5&14&-15&-20&16 \\ \hline
-1&-10&15&21&-25 \\ \hline
-3&-9&13&22&-23 \\ \hline
-2&-8&11&23&-24 \\ \hline
-4&-7&9&24&-22 \\ \hline
-5&-6&7&25&-21 \\ \hline
\end{array}\\
\mbox {Array } C \mbox{ when } m=8&& \mbox{Array } C' \mbox{ when } m=10\\
\end{array}$$
\caption{Arrays $C$ and $C'$ constructed by Proposition \ref{m,5m/2;5,2}}
		\label{C.m=8.5.2}
\end{figure}


\section {The existence of an $SMR(m,n;r,2)$ with $m$ even}\label{SEC4}
Let there exist an $SMR(m,n;r,2)$. If $m=4b$ or $m=4b+2$, then $n=2br$ or $n=(2b+1)r$, respectively.
We study the existence of an $SMR(4b,2br;r,2)$  and an $SMR(4b+2,(2b+1)r;r,2)$ in the
following two subsections, respectively.

\subsection{The existence of an $SMR(4b,2br;r,2)$ }
In this subsection we construct signed magic rectangles with parameters
$(4b,8ab;4a,2)$,
$(4b,2b(4a+2);4a+2,2)$,
$(4b,2b(4a+1);4a+1,2)$, and
$(4b,2b(4a+3);4a+3,2)$, where $a,b\geq 1$.

\begin{lemma}\label{L2,4;4,2}
There exists a shiftable $SMR(2q,4pq;4p,2)$ for positive integers $p,q\geq 1$.
\end{lemma}

\begin{proof}
Figure \ref{2,4;4,2} displays a shiftable $SMR(2,4;4,2)$. So by Part 1 of Theorem \ref{TH:kn;kr-km,kn},
there exists a shiftable $SMR(2,4p;4p,2)$ for $p\geq 1$. Now by Part 2 of Theorem \ref{TH:kn;kr-km,kn} there exists a shiftable $SMR(2q,4pq;4p,2)$ for $p,q\geq 1$.
\end{proof}

\begin{lemma}\label{4b,8ab;4a,2}
There exists a shiftable $SMR(4b,8ab;4a,2)$ for $a,$ $b\geq 1$.
\end{lemma}

\begin{proof}
Apply Lemma \ref{L2,4;4,2} with $p=a$ and $q=2b$ to obtain
a shiftable $SMR(4b,8ab;4a,2)$ for all $a,b\geq 1$.
\end{proof}


\begin{lemma}\label{4b,8ab+4b;4a+2,2}
There exists a shiftable $SMR(4b,2b(4a+2);4a+2,2)$ for $a,b\geq 1$.
\end{lemma}

\begin{proof}
Figure \ref{4,12;6,2} displays a shiftable $SMR(4,12;6,2)$. So by Part 2 of Theorem \ref{TH:kn;kr-km,kn},
there exists a shiftable $SMR(4b,12b;6,2)$, say $A$, for $b\geq 1$. On the other hand, by Lemma \ref{4b,8ab;4a,2}, there exists a shiftable $SMR(4b,8(a-1)b;4(a-1),2)$, say $B$, for $a\geq 2$ and $b\geq 1$.
Now apply Theorem \ref{TH:kn+n';kr+r'} with $A$ and $B$ to obtain
a shiftable $SMR(4b,2b(4a+2);4a+2,2)$ for $a,b\geq 1$.
\end{proof}

\begin{figure}[ht]
$$\begin{array}{|c@{\hspace{0.5mm}}|c@{\hspace{0.5mm}}|c@{\hspace{0.5mm}}
|c@{\hspace{0.5mm}}||c@{\hspace{0.5mm}}|c@{\hspace{0.5mm}}
|c@{\hspace{0.5mm}}|c@{\hspace{0.5mm}}||c@{\hspace{0.5mm}}
|c@{\hspace{0.5mm}}|c@{\hspace{0.5mm}}|c@{\hspace{0.5mm}}|} \hline
-1&2&&&-5&6&&&9&-11&& \\ \hline
1&-2&&&5&-6&&&-9&11&& \\ \hline
&&-3&4&&&-7&8&&&10&-12 \\ \hline
&&3&-4&&&7&-8&&&-10&12 \\ \hline
\end{array}$$
\caption{A shiftable $SMR(4,12;6,2)$}
		\label{4,12;6,2}
\end{figure}

\begin{lemma}\label{4b,8ab+2b;4a+1,2}
There exists an $SMR(4b,2b(4a+1);4a+1,2)$ for $a,b\geq 1$.
\end{lemma}

\begin{proof}
By Proposition \ref{m,5m/2;5,2},
there exists an  $SMR(4b,10b;5,2)$, say $A$, for $b\geq 1$. On the other hand, by Lemma \ref{4b,8ab;4a,2}, there exists a shiftable $SMR(4b,8(a-1)b;4(a-1),2)$, say $B$, for $a\geq 2$ and $b\geq 1$.
Now apply Theorem \ref{TH:kn+n';kr+r'} with $A$ and $B$ to obtain an $SMR(4b,2b(4a+1);4a+1,2)$ for $a\geq 2$ and $b\geq 1$.
When $a=1$ we apply Proposition \ref{m,5m/2;5,2}.
\end{proof}

\begin{lemma}\label{4b,8ab+2b;4a+3,2}
There exists an $SMR(4b,2b(4a+3);4a+3,2)$ for $a,b\geq 1$.
\end{lemma}

\begin{proof}
By Proposition \ref{m,3m/2;3,2},
there exists an  $SMR(4b,6b;3,2)$, say $A$, for $b\geq 1$. On the other hand, by Lemma \ref{4b,8ab;4a,2}, there exists a shiftable $SMR(4b,8ab;4a,2)$, say $B$, for $a, b\geq 1$.
Now apply Theorem \ref{TH:kn+n';kr+r'} with $A$ and $B$ to obtain an $SMR(4b,2b(4a+3);4a+3,2)$ for $a,b\geq 1$.
\end{proof}


\subsection{The existence of an $SMR(4b+2,(2b+1)r;r,2)$}
In this subsection we construct signed magic rectangles with parameters
$(4b+2,2a(4b+2);4a,2)$,
$(4b+2,(2a+1)(4b+2);4a+2, 2)$,
$(4b+2,(4a+1)(2b+1);4a+1,2)$, and
$(4b+2,(4a+3)(2b+1);4a+3,2)$ for all $a,b\geq 1$.

\begin{lemma}\label{2,n,n,2)}
Let $n\equiv 3 \pmod 4$. Then there exists an $SMR(2,n;$ $n,2)$.
\end{lemma}

\begin{proof}
By Lemma \ref{L2,4;4,2}, there exists a shiftable $SMR(2,4k;4k,2)$, say $A$, for $k\geq 1$.
Let $B$ be a $2\times 3$ array with first row $1,2,-3$ and second row $-1,-2,3$. Then $B$ is an $SMR(2,3;3,2)$. Now apply Theorem \ref{TH:kn+n';kr+r'} with $A$ and $B$ to obtain an $SMR(2,4k+3;4k+3,2)$. See Figure \ref{2,11;11,2}.
\end{proof}

\begin{lemma}\label{4b+2,8ab+4a;4a,2}
There exists a shiftable $SMR(4b+2,2a(4b+2);4a,2)$ for $a,b\geq 1$.
\end{lemma}

\begin{proof}
Apply Lemma \ref{L2,4;4,2} with $p=a$ and q=$2b+1$
to obtain a shiftable $SMR(4b+2,2a(4b+2);4a,2)$ for  $a,b\geq 1$.
\end{proof}

\begin{lemma}\label{(4b+2,3(4b+2);6,2)}
There exists a shiftable $SMR(4b+2,3(4b+2);6,2)$ for $b\geq 1$
\end{lemma}

\begin{proof}
Apply Part 2 of Theorem \ref{TH:kn;kr-km,kn} with the shiftable $SMR(4, 12;$ $ 6,2)$ displayed in
Figure \ref{4,12;6,2} to obtain a shiftable $SMR(4(b-1), 12(b-1); 6, 2)$, say $A$.
Then apply Theorem \ref{TH:km+m',kn+n'} with $A$ and the shiftable $SMR(6,18;6,2)$ displayed in
Figure \ref{6,18;6,2} to obtain a shiftable $SMR(4b+2,3(4b+2);6,2)$.
\end{proof}

\begin{figure}[ht]
{\tiny
$$\begin{array}{|c@{\hspace{0.2mm}}|c@{\hspace{0.2mm}}|c@{\hspace{0.2mm}}
|c@{\hspace{0.2mm}}|c@{\hspace{0.2mm}}|c@{\hspace{0.2mm}}
|c@{\hspace{0.2mm}}|c@{\hspace{0.2mm}}|c@{\hspace{0.2mm}}
|c@{\hspace{0.2mm}}|c@{\hspace{0.2mm}}|c@{\hspace{0.2mm}}
|c@{\hspace{0.2mm}}|c@{\hspace{0.2mm}}|c@{\hspace{0.2mm}}
|c@{\hspace{0.2mm}}|c@{\hspace{0.2mm}}|c@{\hspace{0.2mm}}|} \hline
-1&&3&&&&7&-8&&&&&13&-14&&&& \\ \hline
&-2&&4&&&&8&-9&&&&&14&-15&&&  \\ \hline
&&-3&&5&&&&9&-10&&&&&15&-16&& \\ \hline
&&&-4&&6&&&&10&-11&&&&&16&-17& \\ \hline
1&&&&-5&&&&&&11&-12&-13&&&&&18 \\ \hline
&2&&&&-6&-7&&&&&12&&&&&17&-18 \\ \hline
\end{array}$$
}
\caption{A $SMR(6,18;6,2)$}
		\label{6,18;6,2}
\end{figure}

\begin{lemma}\label{4b+2,(2a+1)(4b+2);4a+2, 2}
There exists a shiftable $SMR(4b+2,(2a+1)(4b+2);4a+2, 2)$ for $a,b\geq 1$.
\end{lemma}

\begin{proof}
By Lemma \ref{(4b+2,3(4b+2);6,2)}, there is a shiftable $SMR(4b+2,3(4b+2);6,2)$ for $b\geq 1$, say $A$.
Apply Lemma \ref{L2,4;4,2} with $p=a-1$ and $q=2b+1$
to obtain a shiftable $SMR(2(2b+1),4(a-1)(2b+1);4(a-1),2)$, say $B$, for $a\geq 2$ and $b\geq 1$.
Finally, apply
 Theorem \ref{TH:kn+n';kr+r'} with arrays $A$ and $B$ to obtain a shiftable $SMR(4b+2,(2a+1)(4b+2);4a+2, 2)$ for
 $a\geq 2$ and $b\geq 1$. When $a=1$ apply Lemma \ref{(4b+2,3(4b+2);6,2)}.
\end{proof}

\begin{lemma}\label{4b+2,(4a+1)(2b+1;4a+1,2}
There exists an $SMR(4b+2,(4a+1)(2b+1);4a+1,2)$ for $a,b\geq 1$.
\end{lemma}

\begin{proof}
Apply Lemma \ref{L2,4;4,2} with $p=a-1$ and $q=2b+1$ to obtain
a shiftable $SMR(2(2b+1), 4(a-1)(2b+1); 4(a-1), 2)$, say $A$, for $a\geq 2$.
By Proposition \ref{m,5m/2;5,2} there is an $SMR(4b+2,5(2b+1);5,2)$, say $B$, for $b\geq 1$.
Finally, apply
 Theorem \ref{TH:kn+n';kr+r'} with arrays $A$ and $B$ to obtain an $SMR(4b+2,(4a+1)(2b+1);4a+1,2)$
 for $a,b\geq 1$.
\end{proof}

\begin{lemma}\label{4b+2,(4a+3)(2b+1)(4a+3);4a+3,2}
There exists an $SMR(4b+2,(4a+3)(2b+1);4a+3,2)$ for $a, b\geq 0$.
\end{lemma}

\begin{proof}
Apply Lemma \ref{L2,4;4,2} with $p=a$ and $q=2b+1$ to obtain
a shiftable $SMR(2(2b+1), 4a(2b+1); 4a, 2)$, say $A$.
By Proposition \ref{m,3m/2;3,2} there is an $SMR(4b+2,3(2b+1);3,2)$, say $B$, for $b\geq 1$.
Finally, apply
 Theorem \ref{TH:kn+n';kr+r'} with arrays $A$ and $B$ to obtain an $SMR(4b+2,(4a+3)(2b+1);4a+3,2)$
 for $a,b\geq 1$.
\end{proof}

We conclude this section with the following theorem.
\begin{theorem}\label{mainTHSEC4}
Let $m$ be even. There exists an $SMR(m,n;r,2)$ if and only if either $m=2$ and $n=r\equiv 0,3 \pmod 4$ or
$m\geq 4$, $r\geq 3$ and $mr=2n$.
\end{theorem}

\section {The existence of an $SMR(m,n;r,2)$ with $m$ odd and $r$ even}\label{SEC5}
In this section we investigate the existence of a signed magic rectangle $(m,n;r,2)$ with $m$ odd and $r$ even. Note that if $m$ and $r$ are both odd, then there is no $SMR(m,n;r,2)$.

\subsection {The existence of a shiftable $SMR(m,n;4a,2)$ with $m$ odd}
We consider two cases: $m=4b+1$ and $m=4b+3$.

\begin{lemma}\label{4b+1,2a(4b+1);4a,2}
There exists a shiftable $SMR(4b+1,2a(4b+1);4a,2)$ for all $a,b\geq 1$.
\end{lemma}

\begin{proof}
Apply Lemma \ref{L2,4;4,2} with $p=a=1$ and $q=2(b-1)$ to obtain
a shiftable $SMR(4(b-1),8(b-1);4,2)$ for $b\geq 2$.

Figure \ref{5,10;4,2} displays a shiftable $SMR(5,10; 4,2)$. Therefore there is a shiftable $SMR(4b+1,2(4b+1);4,2)$ by
Theorem \ref {TH:km+m',kn+n'}. Now apply Part 1 of Theorem \ref{TH:kn;kr-km,kn} to obtain a
shiftable $SMR(4b+1,2a(4b+1);4a,2)$ for all $a,b\geq 1$.
\end{proof}

\begin{figure}[ht]
$$\begin{array}{|c|c|c|c|c|c|c|c|c|c|}\hline
1&&&&-5&-6&&&&10 \\ \hline
-1&2&&&&6&-7&&& \\ \hline
&-2&3&&&&7&-8&& \\ \hline
&&-3&4&&&&8&-9& \\ \hline
&&&-4&5&&&&9&-10 \\ \hline
\end{array}$$
\caption{A shiftable $SMR(5,10;4,2)$}
		\label{5,10;4,2}
\end{figure}

\begin{lemma}\label{4b+3,2a(4b+3);4a,2}
There exists a shiftable $SMR(4b+3,2a(4b+3);4a,2)$ for all $a,b\geq 1$.
\end{lemma}

\begin{proof}
Apply Lemma \ref{L2,4;4,2} with $p=1$ and $q=2b$ to obtain
a shiftable $SMR(4b,8b;4,2)$ for $b\geq 1$.
Figure \ref{3,6;4,2} displays a shiftable $SMR(3,6; 4,2)$. Therefore,  by Theorem
\ref{TH:km+m',kn+n'}, there is a shiftable $SMR(4b+3,2(4b+3);4,2)$. We now apply Part 1 of Theorem \ref{TH:kn;kr-km,kn} to obtain a
shiftable $SMR(4b+3,2a(4b+3);4a,2)$ for all $a,b\geq 1$.
\end{proof}

\begin{figure}[ht]
$$\begin{array}{|c|c|c|c|c|c|}\hline
1&&-3&-4&&6 \\ \hline
-1&2&&4&-5& \\ \hline
&-2&3&&5&-6 \\ \hline
\end{array}$$
\caption{A shiftable $SMR(3,6;4,2)$}
		\label{3,6;4,2}
\end{figure}

\subsection {The existence of a shiftable $SMR(m,n;4a+2,2)$ with $m$ odd}

We consider two cases: $m=4b+1$ and $m=4b+3$.

\begin{lemma}\label{4b+1,3(4b+1);6,2}
There exists a shiftable $SMR(4b+1,3(4b+1);6,2)$ for all $b\geq 1$.
\end{lemma}

\begin{proof}
Apply Part 2 of Theorem \ref{TH:kn;kr-km,kn} with the shiftable $SMR(4,12;$ $6,2)$ given in Figure
\ref{4,12;6,2} to obtain a shiftable $SMR(4(b-1),12(b-1);6,2)$ for $b\geq 1$. Figure \ref{5,15;6,2} displays a shiftable $SMR(5,15; 6,2)$. Therefore there is a shiftable $SMR(4b+1,3(4b+1);6,2)$ for $b\geq 1$ by
Theorem \ref{TH:km+m',kn+n'}.
\end{proof}

\begin{figure}[ht]
{\small
$$\begin{array}{|c@{\hspace{0.1mm}}|c@{\hspace{0.1mm}}|c@{\hspace{0.1mm}}
|c@{\hspace{0.1mm}}|c@{\hspace{0.1mm}}|c@{\hspace{0.1mm}}
|c@{\hspace{0.1mm}}|c@{\hspace{0.1mm}}|c@{\hspace{0.1mm}}
|c@{\hspace{0.1mm}}|c@{\hspace{0.1mm}}|c@{\hspace{0.1mm}}
|c@{\hspace{0.1mm}}|c@{\hspace{0.1mm}}|c@{\hspace{0.1mm}}|} \hline
 1&-2&&&&-6&&&&10&&12&&&-15 \\ \hline
 &2&-3&&&6&-7&&&&&&-13&&15 \\ \hline
 &&3&-4&&&7&-8&&&-11&&13&& \\ \hline
 &&&4&-5&&&8&-9&&&-12&&14& \\ \hline
 -1&&&&5&&&&9&-10&11&&&-14& \\ \hline
\end{array}$$
}
\caption{A shiftable $SMR(5,15;6,2)$}
		\label{5,15;6,2}
\end{figure}

\begin{lemma}\label{4b+1,(2a+1)(4b+1);4a+2,2}
There exists a shiftable $SMR(4b+1,(2a+1)(4b+1);4a+2,2)$ for all $a,b\geq 1$.
\end{lemma}

\begin{proof}
Apply Lemma \ref{L2,4;4,2} with $p=1$ and $q=2b-2$ to obtain
a shiftable $SMR(2(2b-2),4(2b-2);4,2)$ for $b\geq 2$.
Figure \ref{5,10;4,2} displays a shiftable $SMR(5,10; 4,2)$. Therefore there is a shiftable $SMR(4b+1,2(4b+1);4,2)$ for $b\geq 1$ by
Theorem \ref{TH:km+m',kn+n'}.
Now  apply Part 1 of Theorem \ref{TH:kn;kr-km,kn} to obtain a
shiftable $SMR(4b+1,2(a-1)(4b+1);4(a-1),2)$, say $A_1$, for all $a\geq 2$ and $b\geq 1$.
By Lemma \ref{4b+1,3(4b+1);6,2} there exists a shiftable $SMR(4b+1,3(4b+1);6,2)$ for $b\geq 1$, say $A_2$. Now apply Theorem \ref{TH:kn+n';kr+r'} with $A_1$ and $A_2$ to obtain a shiftable $SMR(4b+1,(2a+1)(4b+1);4a+2,2)$ for $a\geq 2$ and $b\geq 1$. When $a=1$, we apply Lemma \ref{4b+1,3(4b+1);6,2}.
\end{proof}

\begin{lemma}\label{4b+3,3(4b+3);6,2}
There exists a shiftable $SMR(4b+3,3(4b+3);6,2)$ for all $b\geq 1$.
\end{lemma}

\begin{proof}
Apply Part 2 of Theorem \ref{TH:kn;kr-km,kn} with the shiftable $SMR(4,12;$ $6,2)$ given in Figure
\ref{4,12;6,2} to obtain a shiftable $SMR(4b,12b;6,2)$ for $b\geq 1$. Figure \ref{3,9;6,2} displays a shiftable $SMR(3,9; 6,2)$. Therefore there is a shiftable $SMR(4b+3,3(4b+3);6,2)$ by Theorem \ref{TH:km+m',kn+n'}.
\end{proof}

\begin{figure}[ht]
$$\begin{array}{|c|c|c|c|c|c|c|c|c|}\hline
 1&-2&&-4&&6&7&-8& \\ \hline
&2&-3&4&-5&&-7&&9 \\ \hline
-1&&3&&5&-6&&8&-9 \\ \hline
\end{array}$$
\caption{A shiftable $SMR(3,9;6,2)$}
		\label{3,9;6,2}
\end{figure}

\begin{lemma}\label{4b+3,(2a+1)(4b+3);4a+2,2}
There exists a shiftable $SMR(4b+3,(2a+1)(4b+3);4a+2,2)$ for all $a,b\geq 1$.
\end{lemma}

\begin{proof}
Apply Lemma \ref{L2,4;4,2} with $p=1$ and $q=2b$ to obtain
a shiftable $SMR(2(2b),4(2b);4,2)$ for $b\geq 1$.
Figure \ref{3,6;4,2} displays a shiftable $SMR(3,6; 4,2)$. Therefore there is a shiftable $SMR(4b+3,2(4b+3);4,2)$ by Theorem \ref{TH:km+m',kn+n'}.
Now apply Part 1 of Theorem \ref{TH:kn;kr-km,kn} to obtain a
shiftable $SMR(4b+3,2(a-1)(4b+3);4(a-1),2)$, say $A_1$, for all $a\geq 2$ and $b\geq 1$.
By Lemma \ref{4b+3,3(4b+3);6,2} there exists a shiftable $SMR(4b+3,3(4b+3);6,2)$, say $A_2$, for $b\geq 1$.
Now apply Theorem \ref{TH:kn+n';kr+r'} with $A_1$ and $A_2$ to obtain a shiftable $SMR(4b+3,(2a+1)(4b+3);4a+2,2)$ for $a\geq 2$ and $b\geq 1$. When $a=1$ we apply Lemma \ref{4b+3,3(4b+3);6,2}.
\end{proof}

We summarise the results obtained in Lemmas \ref{4b+1,3(4b+1);6,2}-\ref{4b+3,(2a+1)(4b+3);4a+2,2} in the next theorem.

\begin{theorem}\label{mainTHSEC5}
Let $m$ be odd and $r$ be even. Then there exists a shiftable $SMR(m,n;r,2)$ if and only if $m\geq 3$, $r\geq 4$  and $mr=2n$.
\end{theorem}

We are now ready to state the main theorem of this paper.

\noindent {\bf Main Theorem.} {\em There exists an
$SMR(m,n;r, 2)$ if and only if either $m=2$, $n\equiv 0,3 \pmod 4$ and $n=r$ or
$m,r\geq 3$ and $mr=2n$.}


\end{document}